\documentclass[12pt]{article}
\usepackage{amssymb, url}
\usepackage{amsmath,amsthm}
\usepackage{mathtools}
\usepackage{dsfont}
\usepackage{verbatim}
\usepackage{graphicx}
\usepackage{epstopdf}
\usepackage{fullpage}
\usepackage[hang,flushmargin]{footmisc}
\usepackage{enumitem}
\usepackage{tikz}

\usepackage{datetime}
\usepackage{bbm}

\usepackage{color}

\newtheorem{definition}{Definition}
\newtheorem{theorem}{Theorem}[]
\newtheorem{corollary}[theorem]{Corollary}
\newtheorem{lemma}{Lemma}

\theoremstyle{definition}

\begin{document}
\title{A characterization of the Centers of Chordal Graphs}
\author{James M.\ Shook$^{1,3}$ and Bing Wei$^{2}$}
\date{\today}

\footnotetext[1]{National Institute of Standards and Technology, Computer Security Division, Gaithersburg, MD; {\tt james.shook@nist.gov}. The work of Dr. Shook was supported in part by Graduate Assistance in Areas of National Need (GAANN) while attending The University of Mississippi.}
\footnotetext[2]{University of Mississippi, Oxford, MS; {\tt bwei@olemiss.edu}. The work of B. Wei was supported in part by the Summer Research Grant of College of Liberal Arts at The University of Mississippi.}
\footnotetext[3]{Official Contribution of the National Institute of Standards and Technology; Not subject to copyright in the United States.}
\maketitle
\begin{abstract} A graph is $k$-chordal if it does not have an induced cycle with length greater than $k$. We call a graph chordal if it is $3$-chordal. Let $G$ be a graph. The distance between the vertices $x$ and $y$, denoted by $d_{G}(x,y)$, is the length of a shortest path from $x$ to $y$ in $G$. The eccentricity of a vertex $x$ is defined as $\epsilon_{G}(x)= \max\{d_{G}(x,y)~|~y\in V(G)\}$. The radius of $G$ is defined as $Rad(G)=\min\{\epsilon_{G}(x)~|~x\in V(G)\}$. The diameter of $G$ is defined as $Diam(G)=\max\{\epsilon_{G}(x)~|~x\in V(G)\}$. The graph induced by the set of vertices of $G$ with eccentricity equal to the radius is called the center of $G$. In this paper we present new bounds for the diameter of $k$-chordal graphs, and we give a concise characterization of the centers of chordal graphs.
\end{abstract}
\section{Introduction}

All graphs considered in this paper are finite, simple, and undirected. For terminology and notations not defined here, we refer the reader to \cite{diestel}. We let $V(G)$ and $E(G)$ denote the set of vertices and the set of edges of a graph $G$, respectively. For a vertex set or an edge set $S$ of $G$, we use $|S|$ for the size of $S$. A set $S$ of pairwise adjacent vertices is said to be a clique or an $|S|$-clique. A complete graph on $t$ vertices, denoted by $K_{t}$, is a graph whose vertex set is a $t$-clique. For $S\subseteq V(G)$, we use the notation $\langle S \rangle$ to represent the subgraph induced by $S$ in $G$. For a vertex $x$ or a vertex set $S$ of $G$, we use $G-x$ or $G-S$ for the subgraph of $G$ induced by $V(G)-\{x\}$ or $V(G)-S$, respectively. 
 
The \textbf{distance} between the vertices $x$ and $y$, denoted by $d_{G}(x,y)$, is the length of a shortest path from $x$ to $y$ in $G$. Let $d_{G}(X,S)=\min\{d_{G}(x,s)~|~x\in X,~s\in S\}$ for vertex sets $X$ and $S$ in $G$. We will simply write $d(X,S)$ for $d_{G}(X,S)$ if it is clear that $G$ is under consideration. When $X=\{x\}$, we will sometimes write $d(x,S)$ for $d(\{x\},S)$ and $d(X,S)$. For $S\subseteq V(G)$, we let $N(S)$ be the neighbors of vertices in $S$, $N(x)=N(S)$ when $S=\{x\}$, and $N[S]=N(S)\cup \{S\}$. A set $S\subseteq V(G)$ is said to {\bf dominate} $G$ if every vertex not in $S$ has a neighbor in $S$.  
 
Let $G$ be a graph. The \textbf{eccentricity} of a vertex $x$ is defined as $\epsilon_{G}(x)= \max\{d(x,y)~|~y\in V(G)\}$. The \textbf{radius} of $G$ is defined as $Rad(G)=\min\{\epsilon_{G}(x)~|~x\in V(G)\}$. The \textbf{diameter} of $G$ is defined as $Diam(G)=\max\{\epsilon_{G}(x)~|~x\in V(G)\}$. If $d(x,y)=Diam(G)$ for every pair of vertices $x$ and $y$ in $T\subseteq V(G)$, then we say that $T$ is \textbf{diametrical}. The \textbf{center} of a graph $G$ is the graph induced by $C(G)=\{x~|~\epsilon_{G}(x)=Rad(G)\}$. Let $C^{0}=V(G)$ and $C(G)=C^{1}(G)$, and for $i\geq 2$ let $C^{i}(G)$ denote $C(\langle C^{i-1}(G)\rangle)$. If $V(G)=C(G)$, then $G$ is said to be {\bf self-centered}. When there is no confusion on the graph under consideration we use $R$ for $Rad(G)$, $D$ for $Diam(G)$, $C$ for $C(G)$, and $C^{i}$ for $C^{i}(G)$. 

A $k$-cycle is a cycle with $k$ vertices. A graph is \textbf{$k$-chordal} if it does not have an induced $(k+1)$-cycle. We call a graph chordal if it is $3$-chordal. Note that for every graph $G$ there exists a $k$ such that $G$ is $k$-chordal. 

If $k\geq4$ and $G$ is a $k$-chordal graph, then $G$ is the center of a $k$-chordal graph $G'$ with vertex set $V(G)\cup \{w_{u},w_{v}, u, v\}$ and edge set \[E(G)\cup \{uw_{u},vw_{v}\}\cup \{w_{u}x, w_{v}x~| ~\forall x\in V(G)\}.\] The construction given above doesn't work for chordal graphs because any two non-adjacent vertices in $G$ would form an induced four cycle with $w_{u}$ and $w_{v}$ in the new graph. The goal of this paper is to present a concise characterization of the centers of chordal graphs. In the process we present new bounds for the diameter of $k$-chordal graphs for $k\geq 3$.

A partial characterization of the centers of chordal graphs can be pieced together using results from \cite{OPCCG, kdom, CGCG}. A characterization for chordal graphs similar to ours was given in \cite{chepoi}. However, our approach is different from previous work as we are able to provide some shorter proofs and generalizations of known results. 

In Section~\ref{sec:tstretched} we develop the concept of a $t$-stretched sets and then use it to prove many of our results. For instance, in Section~\ref{sec:tstretched} we prove the following theorem that is a generalization of a theorem found in \cite{kdom}. 

\begin{theorem}\label{T: halfAtleast} If $G$ is a connected $k$-chordal graph, then \[Rad(G)\geq \bigg\lceil \frac{Diam(G)}{2} \bigg\rceil\geq \bigg\lfloor \frac{Diam(G)}{2} \bigg\rfloor \geq Rad(G)-\bigg\lfloor\frac{k}{2}\bigg\rfloor. \] and \[\max\left\{Rad(G),2Rad(G)-2\bigg\lfloor\frac{k}{2}\bigg\rfloor\right\}\leq Diam(G)\leq 2Rad(G).\]
\end{theorem}

The centers of many graphs are disconnected; thus, it is a natural question to ask how far apart the vertices of $C(G)$ in a $k$-chordal graph $G$ are. Using Theorem~\ref{T: halfAtleast} we prove in Section~\ref{sec:tstretched} Theorem~\ref{T: ISCenter}, which generalizes a result given in \cite{CGCG} that showed that the diameter of the center of a chordal graph is at most three. 

\begin{theorem}\label{T: ISCenter}Every connected $k$-chordal graph $G$ has a connected induced subgraph $H$ such that $C(G)\subseteq V(H)$, $Rad(H)\leq 2\big\lfloor\frac{k}{2}\big\rfloor$, and  $Diam(H)\leq 3\big\lfloor\frac{k}{2}\big\rfloor$.
\end{theorem}
From Theorem~\ref{T: ISCenter} we can see that if $G$ is a self-centered $k$-chordal graph, then $Rad(G)\leq 2\big\lfloor\frac{k}{2}\big\rfloor$. We present a characterization of self-centered graphs for $k=3$ a little later. However, we don't know much about self-centered graphs with induced $k$-cycles for $k>3$. A $k$-cycle is an easy example of a self-centered $k$-chordal graph that has radius equal to $\big\lfloor\frac{k}{2}\big\rfloor$. It is a quick exercise to show that non-trivial self-centered $k$-chordal graphs must be $2$-connected. When $k$ is odd the graph with an induced $k$-cycle $x_{1}\ldots x_{k}x_{1}$ along with $k$ vertices $\{y_{1},\ldots,y_{k}\}$ such that $y_{i}$ is adjacent to $x_{i}$, and $x_{i+1}$ modulus $k$ is also self-centered with radius $\big\lfloor\frac{k}{2}\big\rfloor+1$. However, for $k$-chordal graphs with induced $k$-cycles such that $k>3$ we don't know if there exists a self-centered $k$-chordal graph with $Rad(G)=2\big\lfloor\frac{k}{2}\big\rfloor$.  It would be interesting to see a characterization or at the very least some examples for higher $k$.

Our approach to characterizing the centers of chordal graphs is to use Theorem~\ref{T: halfAtleast} to partition the connected chordal graphs into three classes. We show that for each graph $G$ in a partition, the graph $\langle C(G) \rangle$ has a necessary structure. If $Diam(G)=2Rad(G)-1$, then we show that $C(G)$ has two disjoint cliques that each dominates $\langle C(G) \rangle$. If $Diam(G)=2Rad(G)$, then the authors in \cite{CCG2} showed that the center of $G$ is a clique.  If $Diam(G)=2Rad(G)-2$, then we show that $\langle C^{2}(G) \rangle$ is self-centered with radius two. We then show that if a connected chordal graph $G$ with $Diam(G)\leq 3$ has any of those structures, then it is the center of some chordal graph $H$. These necessary and sufficient conditions can be summarized in the following theorem.

\begin{theorem}\label{T: CIFF3} A graph $G$ with at least two vertices is the center of some chordal graph if and only if $G$ is connected, chordal, $Diam(G)\leq 3$ and either $\langle C(G) \rangle$ is self-centered with radius two or $G$ contains two disjoint dominating cliques.\end{theorem}

The smallest graph satisfying Theorem~\ref{T: CIFF3} is $K_{2}$. Note that a $2$-connected chordal graph $G$ with $Rad(G)=2$, $Diam(G)=3$, and $Rad(\langle C(G) \rangle)=1$ that does not have two disjoint dominating cliques cannot be the center of a chordal graph (See Figure~\ref{example}). 

\setlength{\unitlength}{.5 cm}
\begin{figure}[b]
\label{example}
\centering
\begin{tikzpicture}
    \draw (4,4) -- (3,3);
    \draw (4,4) -- (5,3);
    \draw (3,3) -- (5,3);
    \draw (2,2) -- (6,2);
    \draw (2,2) -- (5,3);
    \draw (2,2) -- (3,3);
    \draw (2,2) -- (6,1);
    \draw (2,2) -- (2,1);
    \draw (2,2) -- (1,1);
    \draw (1,1) -- (2,1);
    \draw (2,1) -- (6,1);
    \draw (6,1) -- (6,2);
    \draw (6,1) -- (5,3);
    \draw (6,1) -- (7,1);
    \draw (7,1) -- (6,2);
    \draw (6,2) -- (5,3);
    
    \draw[fill=black] (1,1) circle (.1cm);
    \draw[fill=black] (2,1) circle (.1cm);
    \draw[fill=black] (6,1) circle (.1cm);
    \draw[fill=black] (7,1) circle (.1cm);
    \draw[fill=black] (2,2) circle (.1cm);
    \draw[fill=black] (3,3) circle (.1cm);
    \draw[fill=black] (4,4) circle (.1cm);
    \draw[fill=black] (5,3) circle (.1cm);
    \draw[fill=black] (6,2) circle (.1cm);

\end{tikzpicture}
\caption{This graph is not the center of a chordal graph.}
\end{figure}
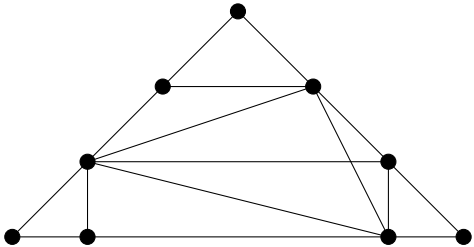

If $G$ is the center of some chordal graph and $Rad(\langle C(G) \rangle)=2$, then either $\langle C(G) \rangle$ is self-centered or $Diam(\langle C(G) \rangle)=3$. By Theorem~\ref{T: halfAtleast} we can see that self-centered chordal graphs have diameter at most two. There are many examples of self-centered chordal graphs without two disjoint dominating cliques. Therefore, it would be nice to have a characterization of chordal graphs whose center has radius two and satisfies the conditions of Theorem~\ref{T: CIFF3}.

Let $T$ and $S$ be sets of vertices in a connected graph $G$. We say $S$ is a \textbf{separator} of $T$ if two vertices of $T$ are disconnected in $G-S$. If $T=V(G)$, then we say $S$ is a separator of $G$. If a vertex set $S$ is a separator of $T$, then we say $S$ separates $T$. A separator of $T$ is said to be \textbf{minimal} if it does not have another separator of $T$ as a proper subset.

A vertex $x$ is said to be \textbf{simplicial} if $N(x)$  is a clique. We use the notation $S_{1}(G)$ to denote the set of simplicial vertices in $G$. A chordal graph is said to be minimally self-centered if the removal of a simplicial vertex results in a chordal graph that is not self-centered. 

In \cite{CGCG}, P. Das and S. B. Rao gave a characterization of minimally self-centered chordal graphs. They showed that for a minimally self-centered chordal graph $G$, $|S_{1}(G)|\geq 3$ and there exists a function $f$ from $S_{1}(G)$ into the set of vertices with degree $n-2$ such that for any $s\in S_{1}(G)$, $s\neq f(s)$ and $sf(s)\notin E(G)$.

In Theorem~\ref{T: selfCentered} we characterize all self-centered chordal graphs.  
\begin{theorem}\label{T: selfCentered}A chordal graph $G$ is self-centered if and only if $G$ is complete or $\Delta(G)\leq n-2$ and there exists a set of cliques $\{X_{i},\ldots,X_{k}\}$ such that
\begin{enumerate}
    \item $X_{i}$ separates $G$ for every $i\in\{1,\ldots,k\}$,
    \item $N(z)\cap X_{i}\neq \emptyset$ for every $z\in V(G)$,
    \item $X_{i}\cap X_{j}\neq \emptyset$ for every $i$ and $j$ in $\{1,\ldots,k\}$,
    \item \[\bigcap_{i\in\{1,\ldots,k\}}X_{i}=\emptyset,\]
    \item for $\{z,z'\}\subseteq V(G)$ there exists a $x\in N(z)\cap N(z')$ such that \[x\in \bigcup_{i\in\{1,\ldots,k\}}X_{i}.\]
\end{enumerate}
\end{theorem}

If $G$ is a chordal graph with $Diam(G)=3$ such that $\langle C(G) \rangle$ is self-centered with radius two, then the cliques given in Theorem~\ref{T: selfCentered} both separates and dominates $G$. We state this formally in the following theorem.

\begin{theorem}\label{T: diamG3} If $G$ is a chordal graph with $Diam(G)=3$ such that $\langle C(G) \rangle$ is self-centered with radius two, then $\Delta(G)\leq n-2$ and there exists a set of at least three pairwise intersecting cliques $\{X_{i},\ldots,X_{k}\}$ of $G$ such that $\bigcap_{i\in\{1,\ldots,k\}}X_{i}=\emptyset$, then every $X_{i}$ separates and dominates $G$.
\end{theorem}

Theorem~\ref{T: diamG3} completes our characterization of the centers of chordal graphs. In short we have shown that every graph that is the center of a chordal graph can be built from a set of dominating cliques.

We give proofs of Theorems~\ref{T: CIFF3}, \ref{T: selfCentered}, and \ref{T: diamG3} in Section~\ref{sec:chordal}. Using the concept of a $t$-stretched set, we are able to deduce with little work the bulk of the previously known results on the centers of chordal graphs. We will reference the original theorems and then give our proofs for completeness.

\section{$t$-stretched sets}\label{sec:tstretched}

For a set $S\subseteq V(G)$, we let $N(S,t)=\{y\in V(G)~|~d(y,S)=t\}$ and $N_{\leq}(S,t)=\{y\in V(G)~|~d(y,S)\leq t\}$. For $t=1$ we have $N(S)=N(S,1)$ and $N[S]=N_{\leq}(S,1)$. We heavily rely on the following concept for the rest of the paper.

\begin{definition}\label{defn:tStretched}Let $T$ be a diametrical set in a graph $G$. For some $t\leq D-1$, we say $T$ is \textbf{$t$-stretched} if for every $u\in T$ the following properties hold:
\begin{enumerate}[label=(C\arabic*),ref=(C\arabic*)]
    \item\label{tStretched_condition:1} $N(u,t)$ does not separate $T-u$,
    \item\label{tStretched_condition:2} Subject to \ref{tStretched_condition:1}, let $X_{u}$ denote a smallest set in $N(u,t)$ that separates $u$ from $T-u$. If $X_{u}$ separates some $w\in V(G)$ from $T-u$ such that $d(w,T-u)=D$ and $d(w,x)>t$ for some $x\in X_{u}$, then $d(w,X_{u})\leq t-1$. 
\end{enumerate}
\end{definition}

If $T$ is a $t$-stretched set, then we let $X_{u}$ denote a smallest set in $N(u,t)$ that separates $u$ from $T-u$, and we denote $W_{u}$ and $W_{T-u}$ to be the components containing $u$ and $T-u$, respectively, in $G-X_{u}$. A $t$-stretched set in $G$ is said to be {\bf maximal} if it  is not a subset of  any other $t$-stretched set of $G$.

To help visualize the definition of a $t$-stretched set, we construct the following example. Let  $a_{0}a_{1}\dots a_{t-1}$, $b_{0}b_{1}\dots b_{t-1}$, and $c_{0}c_{1}\ldots c_{t-1}$ be three induced paths and let $A$ and $B$ be cliques with at least two vertices each. We add an edge from every vertex in $A$ to every vertex in $B$ to form a $(|A|+|B|)$-clique. For some non-empty subset $B'$ of $B$ we add an edge from every vertex of $B'$ to $c_{t-1}$, and then add an edge from $a_{t-1}$ and $b_{t-1}$ to every vertex in $A$ and $B$, respectively. We can see that $\{a_{0}, b_{0}\}$ forms a diametrical pair but is not a $t$-stretched set. On the other hand, $\{a_{0}, c_{0}\}$ is a $t$-stretched set.

We show in the next lemma that we can build a $t$-stretched set from any diametrical set whose vertices satisfy \ref{tStretched_condition:1}. Thus, for every $t$ there exists a diametrical pair that is $t$-stretched.

\begin{lemma}\label{lem:stretchingT}Let $T$ be a diametrical set such that every vertex satisfies \ref{tStretched_condition:1}. If $U\subseteq T$ is all the vertices that satisfy \ref{tStretched_condition:2}, then there exists a $w\in V(G)-V(W_{T-v})$ for some $v\in T-U$ such that $T-\{v\}+\{w\}$ is a diametrical set whose vertices satisfy \ref{tStretched_condition:1} and every vertex in $U\cup \{w\}$ satisfies \ref{tStretched_condition:2}.

\begin{proof}Since $v\in T-U$, there must exist a $w\in V(G)-V(W_{T-v})$ and an $x'\in X_{v}$ such that $d(w,x')\geq t+1$ and $d(w,X_{v})\geq t$. Since $X_{v}-N(w,t)$ is not empty, the component of $V(G)-N(w,t)$ containing $T-v$ contains the vertices of $W_{T-v}$ and is larger. Thus, if we choose such a $w$ so that the component of $V(G)-N(w,t)$ containing $T-v$ is as large as possible, then $w$ satisfies \ref{tStretched_condition:1} and \ref{tStretched_condition:2} with respect to the diametrical set $T-\{v\}+\{w\}$. Since $N(w,t)\cap X_{u}=\emptyset$ for every $u\in U$, we have that every vertex in $U\cup \{w\}$ satisfies \ref{tStretched_condition:1} and \ref{tStretched_condition:2}.
\end{proof}
\end{lemma}

We can also build $t$-stretched sets from smaller ones by just applying Lemma~\ref{lem:stretchingT} multiple times.

\begin{lemma}\label{lem:largerstretchT}Let $T$ be a diametrical set in a graph $G$. If there is a $w\in V(G)$ such that $d(w,T)=D$ and $N(w,t)$ does not separate $T-u$, then there exists a $t$-stretched $T'$ such that $T\subset T'$.
\begin{proof}Note that $T+w$ is diametrical, and every vertex satisfies \ref{tStretched_condition:1}. Furthermore, since every vertex in $T$ satisfies \ref{tStretched_condition:2} we have by Lemma~\ref{lem:stretchingT} that there exists a $t$-stretched set $T'$ such that $T\subset T'$. 
\end{proof}
\end{lemma}

We have following simple observation.
\begin{lemma}\label{lem:basicObservation}For a connected $k$-chordal graph $G$, if $T$ is a diametrical set, then \[C(G)\subseteq \bigcap_{u\in T} N_{\leq}(u,R).\]
\begin{proof}If there exist a $z\in C(G)-N_{\leq}(u,R)$ for some $u\in T$, then we have the contradiction $d(u,z)>R$.
\end{proof}
\end{lemma}

The next lemma gives some basic properties of $t$-stretched sets for $k$-chordal graphs.
\begin{lemma}\label{L: BasicSDS}
Let $G$ be a connected $k$-chordal graph with $t\leq \big\lceil\frac{D}{2}\big\rceil$. If $T$ is a $t$-stretched diametrical set in $G$, then for any $u\in T$, $v\in T-u$, and $x\in X_{u}$ the following properties hold.
\begin{enumerate}[label=(\roman*),ref=(\roman*)]
\item \label{BasicSDS:cond1} $d(x,x')\leq \big\lfloor\frac{k}{2}\big\rfloor$ for every $x'\in X_{u}$.
\item \label{BasicSDS:cond2} $d(x,w)\leq t+\big\lfloor\frac{k}{2}\big\rfloor-1$ for all $w\in V(G)-V(W_{T-u})$.
\item \label{BasicSDS:cond3} If $t\leq \big\lfloor \frac{D}{2}\big \rfloor$, then  $X_{v}\cap V(W_{u})=\emptyset$.
\item \label{BasicSDS:cond4} If $t= \big\lfloor \frac{D}{2} \big\rfloor$, then $d(X_{u},X_{v})=D- 2 \big\lfloor \frac{D}{2} \big\rfloor$.
\item \label{BasicSDS:cond5} If $t=\big\lceil \frac{D}{2} \big\rceil$, then $\epsilon_{G}(x)\leq \big\lfloor \frac{D}{2} \big\rfloor+\big\lfloor\frac{k}{2}\big\rfloor$.
\end{enumerate}

\begin{proof} 
Proof of \ref{BasicSDS:cond1}. Let $x'\in X_{u}-x$. If $xx'$ is an edge, then we are done. We now suppose $xx$ is not an edge. Since $X_{u}$ is a smallest set that separates $u$ and $T-u$, there must exist shortest paths $P$ and $P'$ from $x$ to $x'$ whose internal vertices are in $W_{T-u}$ and $W_{u}$ respectively. Thus, $PP'$ is an induced cycle with length greater than three and at most $k$. This implies that either $l(P)$ or $l(P')$ has degree at most $\big\lfloor\frac{k}{2}\big\rfloor$. Thus, $d(x,x')\leq \big\lfloor\frac{k}{2}\big\rfloor$.

Proof of \ref{BasicSDS:cond2}. Since $T$ is diametrical, we have that $d(v,X_{u}) = D-t$. Thus, for any $w\in V(G)-V(W_{T-u})$ we have that $d(w,X_{u})\leq t$. If $d(w,X_{u})=t$, then $d(x,w)\leq t$. We are left with the case that there exists a path $P$ from $u$ to $x'\in X_{u}$ with length at most $t-1$. By \ref{BasicSDS:cond1} we know that $d(x,x')\leq \big\lfloor\frac{k}{2}\big\rfloor$. Thus, $d(w,x)\leq l(P)+d(x,x')\leq t-1+\big\lfloor\frac{k}{2}\big\rfloor$.

Proof of \ref{BasicSDS:cond3}. If there exists a $v\in T-u$ such that $X_{v}\cap V(W_{u})\neq \emptyset$, then we get a contradiction since $d(v, X_{u})\leq t-1$ implies that $d(u,v)=D< 2t\leq 2\big\lfloor \frac{D}{2} \big\rfloor\leq D$.

Proof of \ref{BasicSDS:cond4}. By the definition of a $t$-stretched set we have  \[D=d(v,u)=d(u,X_{u})+d(X_{u},X_{v})+d(X_{v},v)=2 \bigg\lfloor \frac{D}{2} \bigg\rfloor+d(X_{u},X_{v})\] for every pair $u$ and $v$ in $T$. Solving for $d(X_{u},X_{v})$, we show \ref{BasicSDS:cond4}. 

Proof of \ref{BasicSDS:cond5}. By \ref{BasicSDS:cond2}, for every $w\in V(G)-V(W_{T-u})$ we have that \[d(w,x)\leq \bigg\lceil \frac{D}{2} \bigg\rceil+\bigg\lfloor\frac{k}{2}\bigg\rfloor-1\leq \bigg\lfloor \frac{D}{2}\bigg\rfloor+ \bigg\lfloor\frac{k}{2}\bigg\rfloor.\] On the other hand, every vertex in $W_{T-u}$ is at most $\big\lfloor \frac{D}{2}\big\rfloor$ to some  $x'\in X_{u}$. By \ref{BasicSDS:cond1} we have that $d(x',x)\leq \big\lfloor\frac{k}{2}\big\rfloor$. Thus, \[d(w,x)\leq \bigg\lfloor \frac{D}{2}\bigg\rfloor+ \bigg\lfloor\frac{k}{2}\bigg\rfloor,\] and $\epsilon_{G}(x)\leq \big\lfloor \frac{D}{2}\big\rfloor+\big\lfloor\frac{k}{2}\big\rfloor$ for every vertex $x\in X_{u}$. 
\end{proof}
\end{lemma}

Using Lemma~\ref{L: BasicSDS} we can prove Theorem~\ref{T: halfAtleast} and Theorem~\ref{T: ISCenter}.

\begin{proof}[\textbf{Proof of Theorem~\ref{T: halfAtleast}}] Let $\{u,v\}$ be a $\big\lceil \frac{D}{2} \big \rceil$-stretched diametrical set. If $R < \big\lceil \frac{D}{2}\big\rceil$, then $C(G)\subseteq V(W_{u})$ for every $u\in T$. However, since \[D=\bigg\lceil \frac{D}{2}\bigg\rceil+\bigg\lfloor \frac{D}{2}\bigg\rfloor\] implies $d(v, X_{u})=\big\lfloor \frac{D}{2}\big\rfloor$, we have the contradiction \[d(z,v)=d(z,X_{u})+d(X_{u},v)\geq \bigg\lceil \frac{D}{2}\bigg\rceil\] since 
\[\bigg\lceil \frac{D}{2}\bigg\rceil-1\leq \bigg\lfloor \frac{D}{2}\bigg\rfloor\leq \bigg\lceil \frac{D}{2}\bigg\rceil.\] 

By Lemma~\ref{L: BasicSDS}~\ref{BasicSDS:cond5} we have that $\epsilon_{G}(x)\leq \big\lfloor \frac{D}{2} \big\rfloor+\big\lfloor\frac{k}{2}\big\rfloor$ for all $x\in X_{u}$. Since $R\leq \epsilon_{G}(x)$, we have that $\big\lfloor \frac{D}{2}\big\rfloor\geq R-\big\lfloor\frac{k}{2}\big\rfloor$.

Trivially we have that $R\leq D$, and $D\leq 2Rad(G)$ follows from \[R \geq \bigg\lceil \frac{D}{2}\bigg\rceil\geq \frac{D}{2}.\] Finally, we have that \[D= \bigg\lceil \frac{D}{2}\bigg\rceil+\bigg\lfloor \frac{D}{2}\bigg\rfloor\geq R-\bigg\lfloor\frac{k}{2}\bigg\rfloor+R-\bigg\lfloor\frac{k}{2}\bigg\rfloor=2R-2\bigg\lfloor\frac{k}{2}\bigg\rfloor.\qedhere\]
\end{proof}

\begin{proof}[\textbf{Proof of Theorem~\ref{T: ISCenter}}] Let $u$ be a vertex in a $\big\lceil \frac{D}{2} \big\rceil$-stretched set $T$ and $x\in X_{u}$. Suppose there exists a $z$ in $C(G)$ such that $d(z,X_{u})>\big\lfloor\frac{k}{2}\big\rfloor$. By \ref{BasicSDS:cond5} we have that $R\leq \epsilon(x)\leq \bigg\lfloor\frac{D}{2} \bigg\rfloor+\bigg\lfloor\frac{k}{2}\bigg\rfloor$. If $z\notin W_{u}$, then for some $x\in X_{u}$ we have the contradiction \[R\geq d(z,u)\geq d(z,X_{u})+d(X_{u},u)>\bigg\lfloor\frac{k}{2}\bigg\rfloor+\bigg\lceil\frac{D}{2} \bigg\rceil \geq \epsilon(x)\geq R.\] We now consider the case $z\in W_{u}$. Let $v\in T-u$. Note that $d(v,X_{u})=\big\lfloor\frac{D}{2} \big\rfloor$. Thus, we have the contradiction \[R\geq d(z,v)\geq d(z,X_{u})+d(X_{u},v)>\bigg\lfloor\frac{k}{2}\bigg\rfloor+\bigg\lfloor\frac{D}{2} \bigg\rfloor \geq \epsilon(x)\geq R.\] Thus, $C(G)\subseteq N_{\leq}(X_{u}, \big\lfloor\frac{k}{2}\big\rfloor)$.

Let $H=\langle N_{\leq}(X_{u}, \big\lfloor\frac{k}{2}\big\rfloor)\rangle$. For any two vertices $a$ and $b$ in $H$, there exist vertices $x$ and $x'$ in $X_{u}$ such that $d_{H}(a,x)\leq \big\lfloor\frac{k}{2}\big\rfloor$ and $d_{H}(x',b)\leq \big\lfloor\frac{k}{2}\big\rfloor$. By Lemma~\ref{L: BasicSDS}~\ref{BasicSDS:cond1} we know that every vertex on a shortest path between two vertices of $X_{u}$ is in $N_{\leq}(X_{u}, \big\lfloor\frac{k}{2}\big\rfloor)$. This implies \[d_{H}(a,b)\leq d_{H}(a,x)+d_{H}(x,x')+d_{H}(X',b)\leq 3\bigg\lfloor\frac{k}{2}\bigg\rfloor.\] Thus, $Diam(H)\leq 3\big\lfloor\frac{k}{2}\big\rfloor$. Furthermore, we have that \[Rad(H)\leq d_{H}(x',b)\leq d_{H}(x',x)+d_{H}(x,b)\leq 2\bigg\lfloor\frac{k}{2}\bigg\rfloor. \qedhere\]
\end{proof}

\section{Center of chordal graphs}\label{sec:chordal}

In this section we restrict ourselves to chordal graphs. In \cite{Dirac} it was shown that every minimal separator of a chordal graph is a clique. By Theorem~\ref{T: halfAtleast}, we have $R-1\leq \big\lfloor \frac{D}{2} \big\rfloor \leq  \big\lceil \frac{D}{2}\big\rceil\leq R$ for any chordal graph. The next two lemmas highlight some key relationships between $t$-stretched sets and the centers of chordal graphs.

\begin{lemma}\label{lem:separation} Let $G$ be a connected graph. If $T$ is a $\big\lfloor \frac{D}{2} \big\rfloor$-stretched set, then $X_{u}\cap V(W_{v})=\emptyset$ for every $u\in T$ and $v\in T-u$.
\begin{proof}By contradiction suppose $X_{u}\cap V(W_{v})\neq \emptyset$. Since $X_{u}$ and $X_{v}$ are minimal separations for $u$ and $v$, we may conclude that $X_{u}$ separates some $x\in X_{v}$ from $v$. This implies there is a shortest path $P$ from $v$ to $x$ such that there is a $y\in V(P)\cap (X_{u}-X_{v})$. Thus, $d(v,y)\leq \big\lfloor \frac{D}{2} \big\rfloor-1$, and we have the contradiction \[D=d(u,v)\leq d(u,y)+d(y,v)\leq 2\bigg\lfloor \frac{D}{2} \bigg\rfloor-1<D.\qedhere\]
\end{proof}
\end{lemma}

\begin{lemma}\label{L: BasicC(G)1}If $G$ is a connected chordal graph, and $T$ is a maximal $\big\lfloor \frac{D}{2} \big\rfloor$-stretched set, then $\big\lfloor \frac{D}{2} \big\rfloor = R$ if and only if \[C(G)=\bigcap_{u\in T} X_{u}.\]
\begin{proof}Let \[\hat{X} = \bigcap_{u\in T} X_{u}.\] Suppose $C(G)=\hat{X}$. For $x\in \hat{X}$, $u\in T$, and $v\in T-u$ we have that \[D=d(u,v)\leq d(u,x)+d(x,v)=2\bigg\lfloor \frac{D}{2} \bigg\rfloor\leq \bigg\lceil \frac{D}{2}\bigg\rceil+\bigg\lfloor \frac{D}{2}\bigg\rfloor=D.\] Therefore, \[\bigg\lfloor \frac{D}{2} \bigg\rfloor=\bigg\lceil \frac{D}{2}\bigg\rceil\leq R.\]

For every $x\in \hat{X}$, there is a $w\in V(G)$ such that $d(x,w)=R$. By \ref{BasicSDS:cond1} we know that each $X_{u}$ is a clique. Therefore, if $d(w,X_{u})<\big\lfloor \frac{D}{2} \big\rfloor$ for some $u\in T$, then $R= \big\lfloor \frac{D}{2} \big\rfloor$. On the other hand, if $d(w,X_{u})\geq\big\lfloor \frac{D}{2} \big\rfloor$ for every $u\in T$, then $d(w,T)=D$ since for some $x\in X_{u}$ we have that \[D\geq d(w,u)=d(u,x)+d(x,w)\geq 2\bigg\lfloor \frac{D}{2} \bigg\rfloor =D.\] However, since $\hat{X}$ is not empty, we have that $N(w,\big\lfloor \frac{D}{2} \big\rfloor)$ does not separate $T$. This is a contradiction since Lemma~\ref{lem:stretchingT} implies that $T$ is not maximal. 

Now suppose $\big\lfloor \frac{D}{2} \big\rfloor = R$. Lemma~\ref{lem:separation} states that $X_{u}\cap V(W_{v})=\emptyset$ for all $u\in T$ and $v\in T-u$. Therefore, $V(W_{u})\cap V(W_{v})=\emptyset$. By \ref{lem:basicObservation} we have that $C(G)\subseteq \hat{X}$. Let $z\in \hat{X}$ and $w\in V(G)$. Since $z\in X_{u}$ for all $u\in T$, it must be the case that there is a $v\in T$ such that $w\notin W_{T-v}$. By \ref{BasicSDS:cond2} we have that \[d(w,z)\leq \bigg\lfloor \frac{D}{2} \bigg\rfloor=R.\] Thus, $C(G)=\hat{X}$.
\end{proof}
\end{lemma}

In \cite{CGCG} it was shown with a lengthy proof that the center of a chordal graph has a dominating clique. With Lemma~\ref{L: BasicC(G)1} we can provide not only a shorter proof but a more general result.

\begin{lemma}\label{L: BasicC(G)2} If $G$ is a connected chordal graph and $T$ is a $\big\lceil \frac{D}{2} \big\rceil$-stretched set, then for any $u\in T$, the following properties hold.
\begin{enumerate}[label=(\roman*),ref=(\roman*)]
\item\label{BasicC(G)2_option1} If $\big\lfloor \frac{D}{2} \big\rfloor=Rad(G)$, then $C(G)\subseteq X_{u}$.
\item\label{BasicC(G)2_option2} If $\big\lfloor \frac{D}{2} \big\rfloor=\big\lceil \frac{D}{2} \big\rceil=Rad(G)-1$, then $X_{u}$ dominates $\langle C(G) \rangle$.
\item\label{BasicC(G)2_option3} If $\big\lfloor \frac{D}{2} \big\rfloor<\big\lceil \frac{D}{2} \big\rceil=Rad(G)$, then $X_{u}\subseteq C(G)$ and $X_{u}-X_{v}$ dominates $\langle C(G) \rangle$ for every $v\in T-u$.
\end{enumerate}
\begin{proof}To prove Lemma~\ref{L: BasicC(G)2}~\ref{BasicC(G)2_option1} we just need to recognize that $d(u,X_{v})=D-\big\lceil \frac{D}{2} \big\rceil=R$ for every $u$ and $v\in T-u$. Thus, $\big\lceil \frac{D}{2} \big\rceil=\big\lfloor \frac{D}{2} \big\rfloor$ and $C_{u}\subseteq X_{u}$ follows from Lemma~\ref{L: BasicC(G)1}.

To prove Lemma~\ref{L: BasicC(G)2}~\ref{BasicC(G)2_option2} and \ref{BasicC(G)2_option3}, we begin by choosing an $x\in X_{u}$. By Lemma~\ref{L: BasicSDS}~\ref{BasicSDS:cond5}, we have that $X_{u}\subseteq C(G)$ since \[R\leq \epsilon_{G}(x)\leq \bigg\lfloor \frac{D}{2} \bigg\rfloor+1=R.\] Furthermore, by Lemma~\ref{L: BasicSDS}~\ref{BasicSDS:cond1} we know that $X_{u}$ is complete in $G$. Thus, we need to investigate vertices in $C(G)-X_{u}$.

We now complete the proof of Lemma~\ref{L: BasicC(G)2}~\ref{BasicC(G)2_option2}. Let $z\in C(G)$. We have two cases to consider.  When $z\in N_{\leq}(u,R-1)$ then there is some $v\in T-u$ such that $d(z,X_{u})=1$ since \[R=d(z,v)\geq d(z,X_{u})+d(X_{u},v)\geq d(z,X_{u})+ \bigg\lfloor \frac{D}{2} \bigg\rfloor= d(z,X_{u})+ R-1.\] When $z\notin N_{\leq}(u,R)$, then $d(z,X_{u})=1$ since \[R=d(z,u)\geq d(z,X_{u})+d(X_{u},u)\geq d(z,X_{u})+ \bigg\lfloor \frac{D}{2} \bigg\rfloor= d(z,X_{u})+ R-1.\] Thus, $X_{u}$ dominates $\langle C(G) \rangle$. 

The proof of Lemma~\ref{L: BasicC(G)2}~\ref{BasicC(G)2_option3} follows a similar argument. Let $X'_{u}= X_{u}-X_{v}$ for some $v\in T-u$. By Lemma~\ref{lem:basicObservation} we know that $C-X_{u}\subseteq N_{\leq}(u,R)$. Let $v\in T-u$ and choose a $z\in C-X_{u}$. Since $d(v, X_{u}\cap X_{v})=R$ and $d(v,X_{u})=R-1$, we have that $d(v, X'_{u})=R-1$. Since $d(z,v)\leq R$ and $X_{u}$ separates $z$ from $v$, it must be the case that $d(z,X'_{u})=1$. Thus, $X_{u}'$ dominates $\langle C(G) \rangle$.
\end{proof}
\end{lemma}

Although, the next corollary was proved in \cite{CCG2} we can provide a short proof.
\begin{corollary}\cite{CCG2}\label{C: TR} $C(G)$ is a clique for any connected chordal graph $G$ with $Diam(G)=2Rad(G)$.
\begin{proof}By Lemma~\ref{lem:largerstretchT} there exists a maximal $\frac{D}{2}$-stretched set $T$ in $G$. Since $\frac{D}{2}=R$, we have by Lemma~\ref{L: BasicC(G)1} that $C(G)\subseteq X_{u}$ for every $u\in T$. Since $X_{u}$ is a minimal separation, we know that $C(G)$ is complete.
\end{proof}
\end{corollary}

The following lemma was proved in \cite{CCG2} but is an easy corollary of a lemma given in \cite{CGCG}.

\begin{lemma}~\cite{CCG2, CGCG}\label{L: shortPath} Let $G$ be a connected chordal graph. If $P$ is an induced path in $G$ between two vertices of $C(G)$, then $V(P)\subseteq C(G)$.
\end{lemma}

This implies that $\langle C(G) \rangle$ is connected for every chordal graph $G$.

\begin{lemma}\label{L:TRMO}If $G$ is a connected chordal graph with $Diam(G)=2Rad(G)-1$, then $C(G)$ has two disjoint cliques that each dominates $C(G)$.
\begin{proof} Since $D=2R-1$, we have that $D$ is odd and \[\bigg\lceil \frac{D}{2} \bigg\rceil=\frac{D+1}{2}=R.\] Let $T=\{u,v\}$ be a $R$-stretched pair. By Lemma~\ref{L: BasicC(G)2}~\ref{BasicC(G)2_option3}, both $X_{u}-X_{v}$ and $X_{v}-X_{u}$ dominates $\langle C(G) \rangle$. Since both $X'_{u}$ and $X'_{v}$ are complete in $G$, we have found our desired disjoint cliques.
\end{proof}
\end{lemma}

\begin{lemma}\label{L: domSep}For $t\leq \lceil \frac{D}{2} \rceil$, if $T$ is a $t$-stretched set, then for any $W_{u}$ there exists a $w_{u}\in V(W_{u})$ such that $X_{u}\subseteq N(w_{u})$.
\begin{proof} Let $w_{u}$ be a vertex in $V(W_{u})$ such that $|N(w)\cap X_{u}|$ is maximum. By contradiction we assume there is an  $x\in X_{u}-N(w_{u})$. Since $X_{u}$ is a minimal separation of $u$ and $T-u$ and $W_{u}$ is a component of $G-X_{u}$, there is an induced path $P$ in $W_{u}$ from $w_{u}$ to some $z\in N(x)$. Let $x'\in N(w_{u})\cap X_{u}$. Since $X_{u}$ is a clique, we have that $w_{u}Pxx'$ is a cycle.   
Since $G$ is chordal, the end vertices of an edge of a cycle have at least one common neighbor on the cycle. Let $z'$ be the common neighbor of the edge $xz$. Since $P$ is an induced path and $w_{u}x'$ is an edge, it must be the case that $z'=x'$. Since $x'$ was arbitrarily chosen, we have that $(N(w_{u})\cap X_{u})\cup \{x\} \subset N(z)$. However, this implies the contradiction $|N(w_{u})\cap X_{u}|<|N(z)\cap X_{u}|$.
\end{proof}
\end{lemma}

\begin{lemma}\label{L:TRM2}If $G$ is a connected chordal graph with $Diam(G)=2Rad(G)-2,$ then $\langle C^{2}(G)\rangle$ is self-centered with radius two. Furthermore, for any $(R-1)$-stretched set $T$ of $G$ we have that \[\bigcap_{u\in T}X_{u}=\emptyset\] and for every $u\in T$ and $v\in T-v$ we have that $X_{u}$ dominates $C^{2}(G)$, $X_{u}\cap X_{v}\neq \emptyset$, and there exists $w_{u}\in N_{\leq}(u,R-2)\cap C^{2}(G)$ such that $X_{u}\subseteq N(w_{u})$.

\begin{proof}Let $T$ be a maximal $(R-1)$-stretched diametrical set and let \[\hat{X}=\bigcap_{u\in T}X_{u}.\] By contradiction, suppose $\hat{X}\neq \emptyset$. Since $\frac{D}{2}=R-1$, we have by Lemma~\ref{L: BasicC(G)1} there is an $x\in \hat{X}-C(G)$. This implies there is a $w\in V(G)$ such that $d(w,x)>R$. Note that $d(w,X_{u})\geq R-1$ for all $u\in T$ since $X_{u}$ is complete. By \ref{tStretched_condition:2} we know that $w\not\in N_{\leq}(u,R-1)$ for all $u\in T$. However, this implies $N(w,R-1)$ does not separate $T$ since $x\notin N(w,R-1)$. Furthermore, for all $u\in T$ we have that $d(w,u)=d(w,X_{u})+d(X_{u},u)\geq 2(R-1)=D$. We have a contradiction since Lemma~\ref{lem:largerstretchT} implies $T$ is not maximal. Thus, $\hat{X}=\emptyset$.

For every $u$ and $v$ in $T$, since $d(X_{u},X_{v})=D-2(R-1)=0$, we have by Lemma~\ref{L: BasicSDS}~\ref{BasicSDS:cond4} that $X_{u}\cap X_{v}\neq \emptyset$. However, since $\hat{X}=\emptyset$, we know that $|T|\geq 3$. Let $T=\{u_{1},\ldots,u_{T}\}$. Since $D=2(R-1)$, we know that $d(v,X_{u_{i}})\leq R-1$ for every $u_{i}\in T$ and $v\in V(G)$. By Lemma~\ref{L: domSep} we know that for every $u_{i}\in T$ there is a vertex $w_{u_{i}}\in V(W_{u_{i}})$ that is adjacent to every vertex in $X_{u_{i}}$. Thus, $w_{u_{i}}\in C(G)$ for every such $w_{u_{i}}$. By Lemma~\ref{BasicC(G)2_option2} $X_{u_{i}}$ dominates $\langle C(G)\rangle$. Furthermore, since $\hat{X}=\emptyset$, we may conclude that for every $v\in V(G)$ there is a $u_{i}\in T$ such that $d(v,w_{u_{i}})=2$. Thus, $Rad(\langle C(G) \rangle) = 2$ and $X_{u_{i}}\subseteq C^{2}(G)$. We can also see that $d(w_{u_{i}},v)\leq 2$ for every $u_{i}\in T$ and $v\in V(G)$. Thus, $w_{u_{i}}\in C^{2}(G)$ for every $u_{i}\in T$. Thus, $Rad(\langle C^{2}(G)\rangle)=2$

By Lemma~\ref{L: shortPath} we know that for any $z$ and $z'$ in $C^{2}(G)$ we have that $d_{G}(z,z')=d_{\langle C(G) \rangle}(z,z')=d_{\langle C^{2}(G) \rangle}(z,z')$. Therefore, \[2=Rad(\langle C(G) \rangle)\geq Diam(\langle C^{2}(G)\rangle)\geq Rad(\langle C^{2}(G) \rangle)=2.\] Thus, $\langle C^{2}(G) \rangle$ is self-centered with radius two.
\end{proof}
\end{lemma}

\begin{proof}[\textbf{Proof of Theorem~\ref{T: selfCentered}}] If $G$ is a complete graph, then $G$ is self-centered. If $\Delta(G)\leq n-2$, then $Rad(G)\geq 2$. If every two vertices of $G$ have a common neighbor, then $G$ has diameter two. Thus, $G$ is self-centered.

Suppose $G$ is a self-centered chordal graph that is not complete, then by Theorem~\ref{T: halfAtleast} we have that $Diam(G)=2$. For a be a maximal $1$-stretched set $T$ we have by Lemma~\ref{L: BasicSDS} that $\{X_{u} ~| ~u\in T\}$ satisfies the first four items. Let \[X=\bigcup_{u\in T}X_{u}.\] Suppose there are two vertices $z$ and $z'$ in $V(G)$ that do not share a common neighbor in $X$. If both $z$ and $z'$ are in $X$, then we have a contradiction since $|T|\geq 3$ there exists a $u$ and $v$ in $T$ such that $z$ and $z'$ share a neighbor in $X_{u}\cap X_{v}$. Suppose $z'\notin X$. Since $T$ is maximal, we have by Lemma~\ref{lem:largerstretchT} that $N(1,z')$ does not separates $T$. Thus, $zz'$ must be an edge and there exists a $u\in T$ such that $z\notin X_{u}$. Since $X_{u}$ dominates $G$, there is an $x\in N(z)\cap X_{u}$ and an $x'\in N(Z')\cap (X_{u}-x)$. However, we have a contradiction since $zz'x'x$ would be an induced four cycle. This proves \ref{T: selfCentered}.
\end{proof}

\begin{proof}[\textbf{Proof of Theorem~\ref{T: diamG3}}] Since $Rad(\langle C(G)\rangle)=2$, there exists a set $\Psi$ of cliques in $\langle C(G) \rangle$ that satisfies the conditions of Theorem~\ref{T: selfCentered}. The conditions $|\Psi|\geq 3$ and $\bigcap_{X\in \Psi}X=\emptyset$ carry over from Theorem~\ref{T: selfCentered}. 

Let $X$ be an arbitrary clique in $\Psi$. We need to check that $X$ both separates and dominates $G$.

Suppose $X$ does not separate $G$. Since $X$ separates $\langle C(G) \rangle$, there exists two vertices $z$ and $z'$ in $\langle C(G) \rangle$ that are separated by $X$ in $\langle C(G) \rangle$ but not in $G$. Thus, there exists a shortest path $P$ from $z$ and $z'$ in $G$ such that $V(P-\{z,z'\})\subseteq V(G)-C(G)$. We get a contradiction since Lemma~\ref{L: shortPath} implies every vertex of $P$ is in $C(G)$. Thus, $X$ separates $G$. 

We already know that $X$ dominates $\langle C(G) \rangle$, thus, we need to check the vertices with eccentricity three. Since $X$ separates some $z$ and $z'$ in $C(G)$, and both $z$ and $z'$ have eccentricity two, every vertex must have a neighbor in $X$. Thus, the proof is complete. \end{proof}

The next two lemmas were also proved in \cite{CGCG}. We will provide our proofs for the completeness of the paper.

\begin{lemma}\label{L:RCG2}If $G$ is a connected chordal graph with $Diam(G)\leq 3$ and $Rad(\langle C(G)\rangle)=2$, then $G$ is the center of some chordal graph.
\begin{proof}We proceed by creating a chordal graph that has $G$ as its center. Let $A=\{a_{1},\ldots, a_{|A|}\}$  be the set of vertices of $G$ with eccentricity equal to two. Let $\{w_{1},\ldots, w_{|A|}\}$ be a set of vertices not in $G$. We create a new graph $G'$ by adding the edges $w_{i}a_{i}$ for each $i$. Choose an arbitrary $a_{j}\in A$. The condition $Rad(\langle C(G) \rangle)=2$ implies that $A$ does contain a vertex that dominates $A$. Thus, there is a $a_{k}\in A$ such that $d(a_{j},a_{k})=2$, which implies $d(a_{j},w_{k})=3$ and $d(w_{j},w_{k})=4$. Since $Diam(G)\leq 3$, we have by our choice of $A$ that $\epsilon_{G}(b)=3$ for any $b\in G-A$. Furthermore, since $d_{G}(b, a_{k})\leq 2$ for every $a_{k}\in A$, we have that $d_{G'}(b, w_{k})\leq 3$. For $a_{j}\in A$, we have that $\epsilon_{G'}(a_{j})=3$, $\epsilon_{G'}(w_{j})=4$, and $\epsilon_{G'}(b)=3$ for all $b\in G-A$. Thus, $G$ is the center of $G'$.
\end{proof}
\end{lemma}

\begin{lemma}\label{L:TDDC} If $G$ is a connected chordal graph with $Diam(G)\leq 3$, which contains two disjoint cliques such that each dominates $G$, then $G$ is the center of some chordal graph.
\begin{proof}Let $W_{1}$ and $W_{2}$ be the two disjoint cliques in $G$ that each dominate $G$. For each $i$, add two vertices $w_{i}$ and $w_{i}'$ such that $w_{i}$ is adjacent $w'_{i}$ and to every vertex in $W_{i}$ to create new graph $G'$. For each $i$, let $y\in V(G)-W_{i}$. Since $W_{i}$ dominates $G$, we have that $d(y,w_{i}')=3$. Thus, $\epsilon_{G'}(y)=3$ for every $y\in V(G)$. Furthermore, $d(w_{1},w_{2}')=4$, $d(w_{2},w_{1}')=4$, and $d(w'_{1},w'_{2})=5$. Thus, $Rad(G')=3$, $Diam(G')=5$, and $G$ is the center of $G'$. 
\end{proof}
\end{lemma}

\begin{proof}[\textbf{Proof of Theorem~\ref{T: CIFF3}}]Assume $G$ is the center of some chordal graph $H$. Since every induced subgraph of a chordal graph is chordal, $G$ is chordal. By Theorem~\ref {T: ISCenter}, we have that $Diam(G)\leq 3$. By Theorem~\ref{T: halfAtleast} there are three cases to consider. If $D=2R-1$, then by Lemma~\ref{L:TRMO} $G$ has two disjoint cliques that each dominates $G$. If $D=2R$, then $G$ is a clique with more than two vertices. Choose any two vertices to be the two disjoint cliques. If $D=2R-2$, then by Lemma~\ref{L:TRM2}, $\langle C \rangle$ is self-centered with radius two.

Suppose $G$ is a chordal graph with $Diam(G)\leq 3$.  If $G$ has two disjoint cliques such that each dominates $G$, then by Lemma~\ref{L:TDDC} $G$ is the center of some chordal graph. If $\langle C \rangle$ is self-centered with radius two, then by Lemma~\ref{L:RCG2} $G$ is the center of some chordal graph.
\end{proof}

\end{document}